\documentclass[a4paper,draft]{amsart}
 \usepackage{amsmath,amssymb,amsthm,latexsym}
 \theoremstyle{definition}
 \newtheorem{ddd}{Definition}[section]
 \theoremstyle{plain}
 \newtheorem{ttt}[ddd]{Theorem}
 \newtheorem{llll}[ddd]{Lemma}
 
 \theoremstyle{remark}

 \newcommand{\Ker}{\mathrm{Ker}}
 \newcommand{\opartial}{\overline{\partial}}
 \title[Uniformly distributed  systems]{Uniformly distributed systems of elements on metabelian Lie rings}
  \thanks{This work is financially supported by RFBR (grant 12-01-00084) and by Ministry of  Science of Russian Federation
   (state assignment No\,214/138, project 1052).}
  \author{E.\,N.\,Poroshenko, E.\,I.\,Timoshenko}
  \date{}
 
\begin{document}
   \maketitle
 \begin{abstract}
   In this paper, the notion of a uniformly distributed systems of elements on the variety of metabelian Lie algebras is introduced.
   This notion is analogous to one of a measure preserving systems of elements on group varieties. As the main result of the paper it was
   shown that on the variety of metabelian Lie algebras a system of elements is primitive iff it is uniformly
   distributed.
 \end{abstract}

 \section{Introduction}
   In this paper, we describe systems of elements that uniformly distributed on the variety of Lie rings (over the ring of integers
   $\mathbb Z)$. In
\cite{Pud1,Pud2}, elements uniformly distributed on groups
   were studied. Such elements are called measure preserving. In the papers
\cite{Ti12,Ti2} of the second author, the notion of a
   measure preserving element on a group variety was introduced and
   the measure preserving systems of elements on the varieties of nilpotent and metabelian groups were described.
   It turns out that such systems of elements are primitive i.e. they can be completed to the bases of a free group of the
   corresponding variety.

   Not only the uniform distribution on groups and rings but also other
   distributions can be considered. Some interesting problems
   arise in this case. An example of such problem is to describe the sets of elements
   having the same distribution on a given variety. For this reason, we think that
   it is appropriate to use the terms ``distribution of elements''
   and correspondingly ``distribution of systems of elements'' and although the terms ``measure preserving
   elements'' and ``uniformly distributed elements'' have the same meaning we will use the second one.

   The main result of the paper is Theorem~%
\ref{main}. It claims that a system of  elements of a free
   metabelian Lie ring is uniformly distributed on the variety of metabelian Lie rings iff it is primitive.
   To prove this theorem we had to find a primitivity criterion for this ring variety (Theorem~%
\ref{equivprimunimod}). The proof of this criterion is similar to
   one for the varieties of metabelian groups and metabelian Lie
   algebras, (see.~%
\cite{Rom1, Rom2, Ti3, Ti4, Ti5, Um93}) but it has some
   distinctions.

 \section{Preliminaries}
    In this paper, we assume that
    $X=\{x_1,x_2,\dots , x_n\}$ is a finite set. By
    $L$ denote the free Lie ring with the set of free generators
    $X$ and by
    $M$ denote the free metabelian Lie ring with the same set of generators, i.e.
    free and free metabelian Lie
    $\mathbb{Z}$-algebras respectively. By
    $L'$ and
    $M'$ denote the derived Lie rings of
    $L$ and
    $M$ respectively.

    Let
    $S$ be a Lie ring with generators
    $x_1,x_2,\dots,x_n$. By
    $U(S)$ denote the universal enveloping
    $\mathbb{Z}$-algebra of
    $S$.

    We denote Lie monomials by bracketed lowercase Latin characters and associative monomials by lowercase Latin  characters without brackets.
    Moreover, we use brackets to denote the images of Lie monomials by the natural map from
    $S$ to
    $U(S)$. Namely,
    $[v]$ in
    $U(S)$ is a homogeneous associative polynomial obtained by complete removal of brackets by the rule
    $[[u_1],[u_2]]=[u_1][u_2]-[u_2][u_1]$.

    Let
    $[v]$ be a Lie monomial. By
    $\ell([v])$ denote the length of this monomial i.e. the total number of
    entries of all generators in
    $[v]$.

    Any element
    $g$ in
    $M$ can be considered as a Lie polynomial  in the variables
    $x_1,x_2,\dots, x_n$. So, we denote it by
    $g(x_1,x_2,\dots, x_n)$. Although a representation of an element of
    $M$ is not unique, we will see below that one can use any such representation.
    \begin{ddd}
      A system of  elements
      $\{g_1,g_2,\dots,g_k\}$
      ($k\geqslant 1$) is said to be a \emph{primitive} system of elements of the metabelian  Lie ring
      $M$, if it can be included in a system of free generators of this
      ring.
    \end{ddd}

    Let
    $S^t$ be the direct sum of
    $t$ copies of
    $S$, where
    $S$ is an arbitrary Lie ring. Namely
    $S^t=\underbrace{S\oplus S\oplus\dots \oplus S}_{t \text{ times}}$.
    By
    $p^t$ denote a
    $t$-tuple
    $(p_1,p_2,\dots ,p_t)$ of elements in
    $S$ (then
    $p^t\in S^t$). Similarly, if
    $T$ be a module then by
    $T^t$ denote the module
    $\underbrace{T\oplus T\oplus\dots \oplus T}_{t \text{ times}}$.

    Let
    $R$ be a finite Lie ring of some variety
    $\mathfrak{M}$.  For an arbitrary system of elements
    $\{g_1,g_2,\dots,g_k\}$ in the free Lie ring with
    $n$ generators define the map
    $\psi_{g_1,g_2,\dots, g_k}: R^n\to R^k$ such that
    $$\psi_{g_1,g_2,\dots, g_k}(r^n)=(g_1(r_1,r_2,\dots, r_n),g_2(r_1,r_2,\dots, r_n),\dots g_k(r_1,r_2,\dots, r_n)).$$
    Namely, to obtain
    $\psi_{g_1,g_2,\dots, g_k}(r^n)$ one should substitute
    $r_1, r_2, \dots, r_n$ instead of
    $x_1,x_2,\dots,x_n$ in any representation of
    $g_1,g_2,\dots,g_k$ as Lie polynomials and take the obtained elements in the same order to form
    $k$-tuple. Since any map of generators of a free Lie ring
    $L_{\mathfrak{M}}$ in a variety
    $\mathfrak{M}$ to an arbitrary Lie ring of this variety can be extended to a homomorphism uniquely the
    values of
    $g_1(r^n),g_2(r^n),\dots, g_k(r^n)$ depend on the elements
    $g_1(x_1,x_2,\dots, x_n),g_2(x_1,x_2,\dots, x_n),\dots,g_k(x_1,x_2,\dots, x_n)$ in
    $L_{\mathfrak{M}}$ but do not depend on a polynomials representing these elements.
    Consider the uniform distribution on
    $R^n$. Namely, suppose that each element
    $r_1,r_2,\dots, r_n$ is chosen independently with probability
    $|R|^{-1}$. Then for any
    $r^n\in R^n$ probability to choose it is equal to
    $|R|^{-n}$.
    \begin{ddd}
      Let
      $\mathfrak{M}$ be an arbitrary variety of Lie rings and let
      $R$ be a finite Lie ring in this variety. The system of
      elements
      $\{g_1,g_2,\dots,g_k\}$
      ($k\geqslant 1$) of a free  Lie ring  in a variety
      $\mathfrak{M}$ is called  \emph{uniformly distributed on
      $R$}, if for any
      $p^k\in R^k$ probability that
      $\psi_{g_1,g_2,\dots, g_k}(r^n)=p^k$ is equal to
      $|R|^{-k}$, where
      $r^n$ is chosen at random. It means that if
      $r^n$ runs over
      $R^n$ then any
      $p^k\in R^k$ is the image of exactly
      $|R|^{n-k}$ elements of
      $R^n$.
    \end{ddd}
    \begin{ddd}
      A system of elements
      $\{g_1,g_2,\dots,g_k\}$
      ($k\geqslant 1$) of a free Lie ring of a variety
      $\mathfrak{M}$ is \emph{uniformly distributed on}
      $\mathfrak{M}$ if it is uniformly distributed on any finite
      Lie ring of this variety.
    \end{ddd}

    Clearly, the property of a system of elements to be uniformly distributed on the variety of
    metabelian Lie algebras does not depend on a set of
    free generators chosen in
    $M$. Indeed, let
    $y_1,y_2, \dots y_n$ be some other system of free generators of
    $M$. Then we have
    \begin{equation} \label{twogen}
      \begin{array} {ccc}
        x_1=f_1(y_1,y_2,\dots y_n); & \qquad & y_1=h_1(x_1,x_2,\dots x_n);\\
        x_2=f_2(y_1,y_2,\dots y_n); & \qquad & y_2=h_2(x_1,x_2,\dots x_n);\\
        \dotfill & & \dotfill\\
        x_n=f_n(y_1,y_2,\dots y_n). & \qquad & y_n=h_n(x_1,x_2,\dots x_n).
      \end{array}
    \end{equation}
    Here
    $f_i$ and
    $h_j$ are Lie polynomials. Let
    $$\hat{g}_l(y_1,\dots y_n)=g_l(f_1(y_1,\dots,y_n),\dots, f_{n}(y_1,\dots, y_n)),$$
    and
    $$\check{g}_l(x_1,\dots x_n)=\hat{g}_l(h_1(x_1,\dots x_n),\dots, h_n(x_1,\dots x_n))$$
    for
    $l=1,2,\dots,k$. Then, obviously,
    $g_l(x_1,\dots x_n)=\hat{g}_l(y_1,\dots, y_n)=\check{g}_l(x_1,\dots, x_n)$ in
    $M$.

    Let
    $\mu: R^n\to R^n$ take each
    $r^n\in R^n$ to
    $(s_1,s_2,\dots s_n)=s^n\in R^n$, where
    $s_i=f_i(r^n)$. Let us show that different
    $r^n$'s correspond to different
    $s^n$'s. If it is not the case then the images of some two elements under
    $\mu$ coincide. Since the images of
    $R^n$ are in
    $R^n$ and
    $R^n$ is finite there exists an
    $n$-tuple
    $\check{s}^n=(\check{s}_1,\dots, \check{s}_n)$ not lying in the image of
    $\mu$. Let
    $\check{r}_i=h_i(\check{s}^n)$. Then
(\ref{twogen}) implies
    $\check{s}_i=f_i(\check{r}^n)$, we get a contradiction. Consequently,
    $\mu$ is a bijection. Thus for any
    $p\in R$ the number of
    $n$-tuples
    $r^n\in R^n$ such that
    $(g_1(r^n),g_2(r^n),\dots,g_k(r^n))=p^k$ in
    $R^k$ is equal to the number of
    $n$-tuples
    $s^n\in R^n$, such that
    $(\hat{g}_1(s^n), \hat{g}_2(s^n), \dots,\hat{g}_k(s^n))=p^k$ in
    $R$.

    The group analogues of uniformly distributed
    elements are called measure preserving elements (see~%
\cite{Ti12}). Later on we will need the lemma from
 \cite{Ti12}.
    \begin{llll}\label{abgrlem}
      Let
      $A_n$ be a free abelian group of rank
      $n$. A system of elements
      $\{v_1,v_2,\dots,v_n\}$
      ($1\leqslant k\leqslant n$) of this group preserves measure on the variety of
      abelian groups iff it is primitive.
    \end{llll}

    \begin{ddd} \label{unimod}
      For any associative commutative ring
      $S$ a vector
      $(s_1,s_2,\dots, s_k)\in S^k$ is called
      \emph{unimodular} if the ideal generated by the coordinates of this vector coincides with
      $S$.
    \end{ddd}
    The following definition generalizes Definition~%
\ref{unimod}.
    \begin{ddd}
      Let
      $S$ be an associative commutative ring
      and let
      $I$ be an ideal in this ring. A vector
      $(s_1,s_2,\dots, s_k)\in S^k$ is called
      $I$-\emph{modular} if the ideal generated by the
      coordinates of this vector coincides with
      $I$.
    \end{ddd}
    Let us formulate one more statement we will need in this paper.
    \begin{ttt}\label{art}
\cite{Art76} Let
      $S_0\subset S_1 \subset \dots  S_r \subset \dots$ be a chain of commutative rings satisfying the following
      properties.
      \begin{enumerate}
        \item
          For any
          $r$ the unit of
          $S_r$ lies in
          $S_0$.
        \item
          Any ring
          $S_r$ is a retract of
          $S_{r+1}$, the kernel of this retract is generated by
          an element
          $y_{r+1}$, and
          $y_{r+1}$ is not a zero divisor of the ring
          $\bigcup_{i}S_i$.
        \item
          For any
          $t\geqslant r$ the group
          $GL(t,S_r)$ of invertible matrices of order
          $t$ acts transitively on the set of unimodular vectors in
          $S_r^t$.
        \item
          If
          $I_r$ is the ideal in
          $S_r$ generated by
          $y_1,y_2,\dots, y_r$, then
          $I_r/I_r^2$ is a free
          $S_0$-module of rank
          $r$.
      \end{enumerate}
      Then for all
      $t \geqslant r$ the group
      $GL(t,S_r)$ acts transitively on the set of
      $I_r$-modular vectors in
      $S_r^t$.
    \end{ttt}
    Let us remind the definition of partial derivatives in free and free metabelian Lie rings. Consider the image of
    $f\in L$ in
    $U(L)$ under the natural embedding. For the sake of simplicity let us denote this image also by
    $f$. It is clear that there are unique elements
    $\partial'_1 f, \partial'_2 f, \dots, \partial'_n f\in U(L)$
    such that
    \begin{equation} \label{derivatives}
      f=x_1\partial'_1 f+x_2\partial'_2 f+\dots+x_n \partial'_n f.
    \end{equation}
    These elements are called \emph{partial (right) derivatives} of
    $f$. Evidently, we can say that the maps
    $\partial'_i: L\to U(L)$ are derivations. Namely, these maps have the following properties
    $$\partial'_i(f+g)=\partial'_i f+\partial'_i g;
          \qquad \partial'_i[f,g]=(\partial'_i f)g-(\partial'_i g) f$$
    (more precisely the second property should be written as
    $$\partial'_i[f,g]=(\partial'_i f)\omega(g)-(\partial'_i g) \omega(f),$$
    where
    $\omega: L \to U(L)$ is the natural embedding  of the free Lie ring
    $L$ to
    $U(L)$, i.e.
    $\omega(x_i)=x_i$ and by induction
    $\omega([[u],[v]])=\omega([u])\omega([v])-\omega([v])\omega([u])$).

    Let us define partial derivatives on a free metabelian Lie ring. By
    $\mathbb{Z}[X]$ denote  the set of commutative associative
    polynomials in the variables
    $x_1,x_2,\dots , x_n$. Let
    $\varphi: L\to M$ be the natural homomorphism i.e. the homomorphism
    taking each
    $x_i$ to itself and let
    $\varphi': U(L) \to \mathbb{Z}[X]$ also be the natural homomorphism
    (that is
    $\varphi'(x_i)= x_i$). Consider the maps
    $\partial_i=\varphi'\circ\partial'_i\circ \varphi^{-1}: M\to \mathbb{Z}[X]$.
    It is easy to see that
    $\partial_i$ are well defined. Indeed, the only difficulty in defining
    this map is that
    $\varphi$ is not an isomorphism if
    $|X|>1$. Therefore, each element
    $g\in M$ has more than one pre-image under
    $\varphi$ in
    $L$. However, if
    $g_1,g_2\in L$ are such that
    $\varphi (g_1)=\varphi(g_2)=g$ then
    $g_2-g_1\in \Ker\,\varphi=\langle[L',L']\rangle$, where
    $L'$ is the derivative ring of
    $L$. Therefore, we have
    $g_2-g_1=\sum_{j} \alpha_j[[u_j],[v_j]]$, where
    $[u_j]$ are monomials in
    $[L',L']$,
    $[v_j]$ are monomials in
    $L$, and
    $\alpha_j\in \mathbb{Z}$.
    So, we obtain
    \begin{equation} \label{derker}
      \begin{split}
        \varphi'\circ \partial'_i (g_2-g_1)= & \varphi' \circ \partial'_i \biggl(\sum_j \alpha_j[[u_j],[v_j]]\biggr)\\
          = & \varphi'\biggl( \sum_j \alpha_j \partial'_i([[u_j],[v_j]])\biggr) \\
          = & \varphi' \biggl(\sum_j \alpha_j (\partial'_i([u_j])[v_j]-\partial'_i([v_j])[u_j]) \biggr)\\
          = & \sum_j \alpha_j  (\varphi'\circ\partial'_i([u_j])\varphi'([v_j])-\varphi'\circ\partial'_i([v_j])\varphi'([u_j])).
      \end{split}
    \end{equation}

    Next, let
    $[w]$ be in
    $L'$. Then
    $\varphi'([w])=0$. Indeed,
    $[w]=[[w_1],[w_2]]$ for some Lie monomials
    $[w_1],[w_2]$ in
    $L$ consequently
    $\varphi'([w])=\varphi'([w_1])\varphi'([w_2])-\varphi'([w_2])\varphi'([w_1])$.
    If
    $[w]=[[w_1],[w_2]]$ lies in
    $[L',L']$ then
    $\partial'_i([w])=\partial'_i([w_1])[w_2]-\partial'_i([w_2])[w_1]$.
    Since
    $[w_1],[w_2]\in L'$ we obtain as above that
    $\partial'_i([w])=0$. Therefore, the value of
(\ref{derker}) is
    $0$. So, the value of
    $\partial(g)$ does not depend on the element in
    $\varphi^{-1}(g)$ we are taking.

    Let
    $g$ be an element of a free (metabelian) Lie algebra.
    It follows from the definition of derivatives that the value of a partial derivative of
    $g$ depends on the choice of the system of free generators.

    Let
    $S$ be a metabelian Lie ring and let
    $g(x_1,x_2,\dots,x_n)$,
    $f_1(x_1,x_2,\dots, x_n)$,
    $\dots$,
    $f_n(x_1,x_2,\dots, x_n)$ be Lie polynomials.
    Substitute
    $f_1,f_2,\dots, f_n$ for
    $x_1,x_2,\dots, x_n$ in
    $\partial_i g$. By
    $\partial_i g(f_1,f_2,\dots,f_n)$ denote the obtained
    expression considered as an element of
    $U(S)$.

    Given the system of elements
    $\{g_1,g_2,\dots g_k\}$ of
    $M$ by
    $\mathcal{J}(g_1,g_2,\dots, g_k)$ denote the \emph{Jacobi matrix} of this
    system, i.e. the matrix
    $$(\partial_i g_j)_{n\times k}=
      \begin{pmatrix}
        \partial_1 g_1 & \partial_1 g_2 & \dots &\partial_1 g_k  \\
        \partial_2 g_1 & \partial_2 g_2 & \dots &\partial_2 g_k  \\
        \dots  & \dots & \dots & \dots \\
        \partial_n g_1 & \partial_n g_2 & \dots &\partial_n g_k  \\
      \end{pmatrix}.$$
    Let
    $$f_1(x_1,x_2,\dots,x_n),f_2(x_1,x_2,\dots, x_n),\dots,f_n(x_1,x_2,\dots,x_n)$$
    be Lie polynomials. Substitute these polynomials in
    $\mathcal{J}(g_1,g_2,\dots, g_k)$ for the corresponding
    $x_1,x_2,\dots,x_n$ and denote by
    $\mathcal{J}_{f_1,f_2,\dots, f_n}(g_1,g_2,\dots, g_k)$ the obtained
    matrix.

    Let
    $T$ be the free right
    $\mathbb{Z}[X]$-module with a basis
    $t_1,t_2,\dots t_n$. Consider the set
    $\mathcal{M}$ of square matrices of the second order
    $$
    \begin{pmatrix}
      l    & 0 \\
      \tau & 0
    \end{pmatrix},
    $$
    where
    $l$ is a linear polynomial in
    $\mathbb{Z}[X]$ and
    $\tau\in T$. Define the multiplication on
    $\mathcal{M}$ by the rule
    $[S_1,S_2]=S_1S_2-S_2S_1$. It is easy to show that
    $\mathcal{M}$ is a metabelian Lie ring with respect to this multiplication.

    Let
    $\nu$ be the homomorphism from
    $M$ to
    $\mathcal{M}$ taking each generator
    $x_i$ to the matrix
    $$\begin{pmatrix}
         y_i & 0 \\
         t_i & 0
       \end{pmatrix}.$$
    It is well-known (see, for example,
\cite{Art77, Art76}) that the analogue of
    $\nu$ for Lie algebras over a field (for instance, over
    $\mathbb{Q}$) is an embedding, . Since
    $\mathbb{Z}\subseteq \mathbb{Q}$, this map is obviously an embedding of
    $M$ into
    $\mathcal{M}$.

    By
    $\mathbb{Z}^{1}[X]$ denote the set of linear polynomials (without a free term) in
    variables from the set
    $X$. Next, let
    $p,q,m$ be positive integers and
    $m\geqslant 2$. By
    $I_{p,q,m}$ denote the ideal in
    $\mathbb{Z}[X]$ generated by
    $m$ and
    $x_i^p(x^q-1)$, where
    $i=1,2,\dots,n$. Besides, by
    $\mathbb{Z}_{p,q,m}[X]$ denote the commutative associative ring
    $\mathbb{Z}[X]/I_{p,q,m}$, by
    $\mathbb{Z}^1_{m}[X]$ the set of linear polynomials in
    $\mathbb{Z}_{p,q,m}[X]$, and by
    $T_{p,q,m}$ the free right
    $n$-generated
    $\mathbb{Z}_{m}$-module. To denote the elements of a basis of
    $T_{p,q,m}$ we use the characters
    $t_1,t_2,\dots,t_n$ as well as to denote the elements of the basis of
    $T$. Finally, by
    $\mathcal{M}_{p,q,m}$ we denote the Lie ring of
    $2\times 2$ matrices of the form
    $$
    \begin{pmatrix}
      \overline{l}     & 0\\
      \overline{\tau}  & 0
    \end{pmatrix},
    $$
    where
    $\overline{l}\in\mathbb{Z}_{p,q,m}[X]$,
    $\overline{\tau} \in T_{p,q,m}$, and the Lie multiplication is defined in the natural way.
    Obviously,
    $\mathcal{M}_{p,q,m}$ is a finite metabelian Lie ring.

    Let
    $\eta: \mathbb{Z}[X]\to \mathbb{Z}_{p,q,m}[X]$ be
    the natural homomorphism. By
    $\opartial_i g$ denote the image of
    $\partial_i g$ under the action of
    $\eta$. Respectively, by
    $\overline{\mathcal{J}}_{p,q,m}(g_1,g_2,\dots,g_k)$ let us denote the
    Jacobi matrix for a system of elements
    $\{g_1,g_2,\dots,g_k\}$ over
    $\mathbb{Z}_{p,q,m}[X]$, i.e. the matrix
    $$(\opartial_ig_j)_{n\times k}=
      \begin{pmatrix}
        \opartial_1 g_1 & \opartial_1 g_2 & \dots &\opartial_1 g_k  \\
        \opartial_2 g_1 & \opartial_2 g_2 & \dots &\opartial_2 g_k  \\
        \dots  & \dots & \dots & \dots \\
        \opartial_n g_1 & \opartial_n g_2 & \dots &\opartial_n g_k  \\
      \end{pmatrix}.$$

    The following diagram shows the relationship among the rings described
    above:
    \begin{center}
      \begin{picture}(160,70)
         \put(0,50){$L$}
         \put(30,50){$U(L)$}
         \put(75,50){$\mathbb{Z}[X]$}
         \put(120,50){$\mathbb{Z}_{p,q,m}[X]$}
         \put(10,54){\vector(1,0){16}}
         \put(55,54){\vector(1,0){16}}
         \put(100,54){\vector(1,0){16}}
         \put(4,46){\vector(0,-1){36}}
         \put(-2,0){$M$}
         \put(10,7){\vector(2,1){72}}
         \put(10,4){\vector(3,1){115}}
         \put(14,58){{\footnotesize$\partial'_i$}}
         \put(59,58){{\footnotesize$\varphi'$}}
         \put(104,58){{\footnotesize$\eta$}}
         \put(-4,28){{\footnotesize$\varphi$}}
         \put(37,28){{\footnotesize$\partial_i$}}
         \put(67,15){{\footnotesize$\opartial_i$}}
      \end{picture}
    \end{center}

  \section{Property of primitivity and uniform distribution of elements}
    Let
    $g\in M$. Later on, by
    $g'$ we denote the element of
    $M$ such that
    $g-g'$ is a linear combination of the elements in
    $X$. This linear combination itself will be denoted by
    $\overline{g}$. The identities of a free metabelian Lie ring are
    homogeneous, therefore, for any
    $g$ the elements
    $\overline{g}$ and
    $g'$ are defined uniquely and
    $g=\overline{g}+g'$. By
    $\Delta$ denote the ideal in
    $\mathbb{Z}[X]$ generated by the set
    $X$. Finally, for any matrix
    $\mathbf{A}=(a_{ij})_{n\times n}$ we will use the following notation
    $$\sigma_i(\mathbf{A})=\sum_{j=1}^n x_ja_{ji}.$$

    \begin{llll} \label{sigma}
      Let
      $\mathbf{A}$ be an invertible
      $n\times n$ matrix such that its coefficients are in
      $\mathbb{Z}[X]$. Then
      $\sigma_1(\mathbf{A}), \sigma_2(\mathbf{A}), \dots, \sigma_n(\mathbf{A})$ generates the ideal
      $\Delta$.
    \end{llll}
    \begin{proof}
      The statement of this lemma follows from the following chain of equalities.
      $$(x_1,x_2,\dots, x_n)=(x_1,x_2,\dots, x_n)\mathbf{A}\mathbf{A}^{-1}=
        (\sigma_1(\mathbf{A}),\sigma_2(\mathbf{A}),\dots, \sigma_n(\mathbf{A}))\mathbf{A}^{-1}.$$
    \end{proof}
    \begin{llll}\label{automorph}
      Let
      $\mu: M\to M$ be an automorphism of
      $M$ such that
      $\mu(x_i)=g_i$. Then the Jacobi matrix
      $\mathcal{J}(g_1,g_2,\dots,g_n)$ is invertible in the ring of
      $n\times n$ matrices over
      $\mathbb{Z}[X]$.
    \end{llll}
    \begin{proof}

      Let
      $\mu_1$ and
      $\mu_2$ be endomorphisms of the free metabelian Lie ring
      $M$ defined as follows:
      $\mu_j(x_i)=y_{j,i}(x_1,x_2,\dots,x_n)$. Then we have
      \begin{equation*}
        \mu_2\circ \mu_1(x_i)=\mu_2(\mu_1(x_i))=y_{1,i}(y_{2,1}(x_1,\dots,x_n),\dots,y_{2,n}(x_1,\dots,x_n)).
      \end{equation*}
      By
      $z_{j}(x_1,x_2,\dots,x_n)$ denote
      $y_{1,j}(y_{2,1}(x_1,\dots,x_n),\dots,y_{2,n}(x_1,\dots,x_n))$. The fol\-lo\-wing
      ``chain-rule'' formulas hold:
      \begin{equation*}
        \partial_i z_j =\sum_{k=1}^{n} \partial_iy_{2,k}
         \partial_k y_{1,j}(\overline{y}_{2,1},\overline{y}_{2,2},\dots, \overline{y}_{2,n}).
      \end{equation*}
      These formulas imply
      \begin{equation}\label{chainrule}
        \mathcal{J}(z_1,z_2,\dots,z_n)=\mathcal{J}(y_{2,1},y_{2,2},\dots,y_{2,n})
           \mathcal{J}_{\overline{y}_{2,1},\overline{y}_{2,2},\dots,\overline{y}_{2,n}}(y_{1,1},y_{1,2},\dots,y_{1,n}).
      \end{equation}

      Let
      $\mu$ be an automorphism of
      $M$. Then
      $\mu^{-1}$ is also an automorphism. Let
      $\mu^{-1}(x_i)=f_i$. Using
(\ref{chainrule}) for the identity automorphism we obtain
      \begin{equation}\label{chainruleinverse}
        \mathbf{E}=\mathcal{J}(x_1,x_2,\dots,x_n)=\mathcal{J}(g_1,g_2,\dots,g_n)
          \mathcal{J}_{\overline{g}_1,\overline{g}_2,\dots,\overline{g}_n}(f_1,f_2,\dots,f_n).
      \end{equation}
      Note that the  elements of
      $\mathcal{J}_{\overline{g}_1,\overline{g}_2,\dots,\overline{g}_n}(f_1,f_2,\dots,f_n)$
      are polynomials with integer coefficients. It implies that
      the matrix
      $\mathcal{J}(g_1,g_2,\dots,g_n)$ is invertible from the right in the ring of matrices over
      $\mathbb{Z}[X]$. By
\cite{Um93} the Jacobi matrix
      $\mathcal{J}(g_1,g_2,\dots,g_n)$ is invertible in the ring of matrices over
      $\mathbb{Q}[X]$. Consequently, the right and left inverses of this matrix coincide.
      Thus,
      $\mathcal{J}(g_1,g_2,\dots,g_n)$ is invertible in the ring of matrices over
      $\mathbb{Z}[X]$.
    \end{proof}

    \begin{ttt}\label{equivprimunimod}
      A system of elements
      $\{g_1,g_2,\dots,g_k\}$
      ($1\leqslant k\leqslant n$) of the free metabelian Lie ring
      $M$ is primitive iff the ideal generated by all
      $k\times k$ minors of
      $\mathcal{J}(g_1,g_2, \dots, g_k)$ coincides with
      $\mathbb{Z}[X]$.
    \end{ttt}
    \begin{proof}
      Suppose that the ideal generated by all
      $k\times k$ minors of
      $\mathcal{J}(g_1,g_2,\dots, g_k)$ coincides with
      $\mathbb{Z}[X]$. Then each column of
      $\mathcal{J}(g_1,g_2,\dots,g_k)$ is a unimodular vector. Consequently,
\cite{Su76} implies that there exists an invertible
      $n\times n$ matrix
      $\mathbf{B_1}$ with entries in
      $\mathbb{Z}[X]$ and such that
      $\mathbf{B_1}\cdot (\partial_1 g_1,\partial_2 g_1, \dots, \partial _n g_1)^t=(1,0,\dots,0)^t$.
      Hence
      $$
        \mathbf{B_1}\mathcal{J}(g_1,g_2,\dots,g_k)=
        \begin{pmatrix}
          1 & * & \dots & * \\
          0 & * & \dots & * \\
          \vdots & \vdots &  \ddots &  \vdots \\
          0 & * & \dots & *
        \end{pmatrix}
      $$
      It can  be shown that
      $k\times k$ minors of
      $\mathbf{B_1}\mathcal{J}(g_1,g_2,\dots,g_k)$ are linear
      combinations of the
      $k\times k$ minors of
      $\mathcal{J}(g_1,g_2,\dots,g_k)$. So, the ideal  generated by
      $k\times k$ minors of
      $\mathbf{B_1}\mathcal{J}(g_1,g_2,\dots,g_k)$  is  contained in the ideal generated by the
      $k\times k$ minors of
      $\mathcal{J}(g_1,g_2,\dots,g_k)$. Since
      $\mathbf{B_1}$ is invertible,  these ideals coincide.

      Let
      $\mathbf{J_1}$ be obtained from
      $\mathbf{B_1} \mathcal{J}(g_1,g_2,\dots,g_k)$ by
      deletion the first row and the first column.
      Then the set of
      $(k-1)\times (k-1)$ minors of
      $\mathbf{J_1}$ coincides with
      $\mathbb{Z}[X]$. Therefore the first column of
      $\mathbf{J_1}$ is unimodular. Thus there exist
      $(n-1)\times (n-1)$ matrix
      $\mathbf{B'_2}$ such that the  product of
      $\mathbf{B'_2}$ by the first column of
      $\mathbf{J_1}$ is equal to
      the column
      $(1,0,\dots,0)^t$ of the length
      $k-1$. extend
      $\mathbf{B'_2}$ up to an
      $n\times n$ matrix
      $\mathbf{B_2}$ as follows:
      $$
      \mathbf{B_2}=
      \begin{pmatrix}
        1 &
           \begin{matrix}
             0 &  \dots & 0
           \end{matrix}\\
        \begin{matrix}
         0 \\
         \vdots\\
         0
        \end{matrix}&
        \boxed{
        \begin{matrix}
           &  & \\
           & \mathbf{B'_2} & \\
           & &
        \end{matrix}}
      \end{pmatrix}
      $$

      We obtain
      $$
      \mathbf{B_2B_1}\mathcal{J}(g_1,g_2,\dots, g_k)=
      \begin{pmatrix}
        1      & *      & * &  \dots & * \\
        0      & 1      & * &  \dots & * \\
        0      & 0      & * &  \dots & * \\
        \vdots & \vdots & \vdots & \ddots & \vdots \\
        0      & 0      & * &  \dots & *
      \end{pmatrix}
      $$
      And so on. Finally we obtain a matrix of the form
      $$\mathbf{B_kB_{k-1}\dots B_1}\mathcal{J}(g_1,g_2,\dots,g_k)=
        \begin{pmatrix}
          1      & *      & * &  \dots & * \\
          0      & 1      & * &  \dots & * \\
          0      & 0      & 1 &  \dots & * \\
          \vdots & \vdots & \vdots & \ddots & \vdots \\
          0      & 0      & 0 &  \dots & 1\\
          \vdots & \vdots & \vdots & \ddots & \vdots \\
          0      & 0      & 0 &  \dots & 0
        \end{pmatrix}
      $$
      Extend this matrix up to an
      $n\times n$ matrix as follows:      $$
        \mathbf{\widehat{C}}=
        \begin{pmatrix}
        \boxed{
          \begin{array}{c}
            \\
            \\
            \\
            \mathbf{B_k B_{k-1}\dots B_1}\mathcal{J}(g_1,g_2,\dots, g_k)\\
            \\
            \\
            \\
          \end{array}}
          \begin{array}{cccc}
            0 & 0 &\dots & 0 \\
            \vdots& \vdots& \ddots & \vdots \\
            0 & 0 &\dots & 0 \\
            1 & 0 &\dots & 0 \\
            0 & 1 & \dots &  0 \\
            \vdots& \vdots& \ddots & \vdots \\
            0 & 0 &  \dots &  1
          \end{array}
        \end{pmatrix}
      $$
      Clearly,
      $\mathbf{\widetilde{C}}=(\mathbf{B_kB_{k-1}\dots B_1})^{-1}\mathbf{\widehat{C}}$
      is of the form
      $$\mathbf{\widetilde{C}}=
        \left(
        \begin{array}{c|c}
          \mathcal{J}(g_1,g_2,\dots,g_k)
          &
          \mathbf{C'}
        \end{array}
        \right)
      $$
      for some
      $n\times (n-k)$ matrix
      $\mathbf{C'}$.

      It is obvious that
      $\mathbf{\widetilde{C}}$ is invertible. Moreover, let us notice that
      $\overline{g}_j=\sigma_j(\mathbf{\widetilde{C}})$ for
      $j=1,2,\dots,k$.

      Consider the homomorphism
      $\varepsilon:\mathbb{Z}[X]\to \mathbb{Z}$ that takes each
      polynomial in
      $\mathbb{Z}[X]$ to its free term. The ideal
      generated by all
      $k\times k$ minors of
      $\mathcal{J}(g_1,g_2,\dots,g_k)$ coincides with
      $\mathbb{Z}[X]$. Therefore applying
      $\varepsilon$ to the minors of
      $\mathcal{J}(g_1,g_2,\dots,g_k)$  we obtain that the ideal generated by all
      $k\times k$ minors of
      $\mathcal{J}(\overline{g}_1,\overline{g}_2,\dots,\overline{g}_k)$ coincides with
      $\mathbb{Z}$.  It is  well known that in this  case the system
      $\{\overline{g}_1,\overline{g}_2,\dots,\overline{g}_k\}$
      can be included in a set of generators of the free abelian Lie algebra
      with a basis
      $X$. Therefore, this system can be included in a system of generators  of the
      free metabelian Lie algebra
      $M$ as well.

      Let
      $R$ be a subring of
      $\mathbb{Z}[X]$ such that it is generated by
      $\overline{g}_{k+1}, \overline{g}_{k+2},\dots, \overline{g}_n$.
      Consider the elements
      $\sigma_{k+1}(\mathbf{B}^{-1}), \sigma_{k+2}(\mathbf{B}^{-1}), \dots, \sigma_n(\mathbf{B}^{-1})$.
      they can be represented as polynomials in
      $\overline{g}_1, \overline{g}_2,\dots, \overline{g}_n$. Thus, one can find linear combinations
      $\widetilde{g}_{k+1},\widetilde{g}_{k+2},\dots ,\widetilde{g}_{n}$  of the elements
      $\sigma_{1}(\mathbf{\widetilde{C}}), \sigma_{2}(\mathbf{\widetilde{C}}), \dots, \sigma_k(\mathbf{\widetilde{C}})$,
      with coefficients in
      $\mathbb{Z}[X]$
      such that these combinations satisfy following property. If
      we represent
      $\sigma_i(\mathbf{\widetilde{C}})-\widehat{g}_i$
      ($k+1\leqslant i \leqslant n$) as polynomials in
      $\overline{g}_1, \overline{g}_2,\dots, \overline{g}_n$ then these polynomials do not depend on
      $\overline{g}_{1},\overline{g}_{2},\dots, \overline{g}_k$. For all
      $i=k+1,k+2,\dots,n$ subtract the linear combinations corresponding to
      $\widehat{g}_{i}$ from the last
      $i$th columns of
      $\mathbf{\widetilde{C}}$. Let
      $\mathbf{C}$ be the obtained matrix.

      Denote by
      $\mathbf{E}_s$ the
      $s\times s$ identity matrix and by
      $\mathbf{E}_{s,i,j}$ the
      $s\times s$ matrix unit corresponding  to the
      $i$th row and the
      $j$th column. Then we have
      $\mathbf{C}$ is a product of
      $\mathbf{\widetilde{C}}$ and the matrices of the form
      $\mathbf{E}-\widehat{g}_i\mathbf{E}_{i,j}$ for
      $i=1,2,3,\dots, k$ and
      $j=k+1,k+2,\dots, n$. Since the matrices
      $\mathbf{\widetilde{C}}$ and
      $\mathbf{E}-\widehat{g}_i\mathbf{E}_{i,j}$ are invertible  so is
      $\mathbf{C}$. The elements
      $\sigma_1(\mathbf{C})$,
      $\sigma_2(\mathbf{C})$,
      $\dots$,
      $\sigma_n(\mathbf{C})$ generate the same ideal in
      $\mathbb{Z}[X]$ as the elements
      $\overline{g}_1,\overline{g}_2,\dots, \overline{g}_n$ do. At it was shown above, this ideal is
      $\Delta$. Consequently,
      $\sigma_{k+1}(\mathbf{C})$,
      $\sigma_{k+2}(\mathbf{C})$,
      $\dots$,
      $\sigma_n(\mathbf{C})$ generate the same ideal of
      $R$ as
      $\overline{g}_{k+1},\overline{g}_{k+2}, \dots, \overline{g}_n$ do.

      Consider the following chain of the subrings of
      $R$:
      $\mathbb{Z}\subset\mathbb{Z}[\overline{g}_{k+1}]\subset \mathbb{Z}[\overline{g}_{k+1},\overline{g}_{k+2}]
        \subset \dots \subset \mathbb{Z}[\overline{g}_{k+1},\overline{g}_{k+2},\dots \overline{g}_n]=R$.
      Easy to see, that this chain satisfies the conditions of Theorem~%
\ref{art}. Indeed, the first two conditions are obvious (in the
      second condition we take
      $\overline{g}_{k+i}$ for
      $y_i$). The third condition follows from
\cite{Su76}. The fourth condition is also obvious because
      $I_{r}=\langle\overline{g}_{k+1},\overline{g}_{k+2}, \dots, \overline{g}_{k+r}\rangle$, therefore
      $I_r/I_r^2$ is isomorphic to the set of linear combinations of
      $\overline{g}_{k+1},\overline{g}_{k+2}, \dots, \overline{g}_{k+r}$.
      This means that
      $I_r/I_r^2$ is a
      $\mathbb{Z}$-module of rank
      $r$. Consequently, the set of all
      $(n-k)\times(n-k)$-matrices with coefficients in
      $\mathbb{Z}[\overline{g}_{k+1},\overline{g}_{k+2}, \dots,\overline{g}_{n}]$
      acts transitively on the set of all
      $I_{n-k}$-modular vectors. In other words, there exists an invertible matrix
      $\mathbf{D}_1$ (in which the elements
      $\overline{g}_{k+1},\overline{g}_{k+2}, \dots,\overline{g}_{n}$ are written as expressions in
      $x_1,x_2,\dots, x_n$) such that
      $(\sigma_{k+1}(\mathbf{C}),\sigma_{k+2}(\mathbf{C}),\dots, \sigma_n(\mathbf{C}))\mathbf{D}_1=
       (\overline{g}_{k+1},\overline{g}_{k+2},\dots, \overline{g}_n)$.

      Let
      $$\mathbf{D}=
        \left(
        \begin{array}{c|ccc}
          \mathbf{E}_{k} &  0 & \dots & 0\\ \hline
            0  &   &  & \\
            \vdots & & \mathbf{D_1} & \\
            0  &   &  & \\
       \end{array}
       \right).$$
      Then
      $$(x_1,x_2,\dots, x_n)\mathbf{C}\mathbf{D}=(\sigma_1(\mathbf{C}),\sigma_2(\mathbf{C}),\dots, \sigma_n(\mathbf{C}))\mathbf{D}=
        (\overline{g}_1,\overline{g}_2,\dots,\overline{g}_n).$$
      Therefore, we obtain
      $\sigma_{i}(\mathbf{C}\mathbf{D})=\overline{g}_i$ for all
      $i=1,2,\dots, n$. There is a unique representation
      \begin{equation}\label{lin-nonlin}
         \mathbf{C}\mathbf{D}=\mathbf{\overline{F}}+\mathbf{F'},
      \end{equation}
      where
      $\mathbf{\overline{F}}$ is such that its elements are integers and
      $\mathbf{F'}$ is such that its elements are polynomials in the variables from the set
      $X$ and without free terms.

      By computing degrees of polynomials in both sides of
(\ref{lin-nonlin}) one can obtain that
      $\sigma_{i}(\mathbf{\overline{F}})=\overline{g}_i$. Therefore,
      $\sigma_{i}(\mathbf{F'})=0$. Thus,
\cite{Um93} implies that there exist elements
      $g_1',g_2',\dots, g_n'\in M_{\mathbb{Q}}'$ such that
      $\mathbf{CD}=\mathcal{J}(g_1,g_2,\dots, g_n)$, where
      $g_i=\overline{g}_i+g_i'$ for
      $i=1,2,\dots n$
      ($M_{\mathbb{Q}}$ is the free metabelian Lie
      $\mathbb{Q}$-algebra with the set of generators
      $X$).

      Let
      $[u]=[\dots[[x_{i_1}, x_{i_2}]x_{i_3}]\dots x_{i_k}]\in M'$. It is easy to show by induction that
      \begin{equation}\label{diff}
        \partial_{i}[u]=
           \begin{cases}
             x_{i_2}x_{i_3}\dots x_{i_k}, & \text{if } i=i_1 \\
            -x_{i_1}x_{i_3}\dots x_{i_k}, & \text{if } i=i_2 \\
            0,                             & \text{else}.
          \end{cases}
      \end{equation}

      Let us put a linear order on
      $X$. By
\cite{Bo63,Shme64}, the set of all right-normed elements of the
      form
      $[\dots[[x_{i_1}, x_{i_2}]x_{i_3}]\dots x_{i_k}]$ where
      $x_{i_2}<x_{i_1}$,
      $x_{i_2} \leqslant x_{i_3}\leqslant \dots \leqslant x_{i_k}$ is a basis of
      $M$. Later on, such monomials will be called \emph{basis} ones.

      Suppose that
      $h=\alpha[u]$ for some
      $[u]\in M'$ starting with
      $x_i$. It follows from
(\ref{diff}) that
      $\partial_i h=\alpha \partial_i [u]$. Consider distinct Lie monomials
      $[u]$ such that
      $x_i$ is not the least generator in these monomials. Then the non-zero values
      of monomials
      $\partial_i [u]$ are also distinct. Let
      $h$ be an element of
      $M'$ such that
      $h=\sum\alpha_j [u_j]$, where
      $[u_j]\in M'$ and let
      $[u_j]$ start with
      $x_{i_j}$. If we write derivatives
      $\partial_{i} h$ as a linear combinations of
      $\partial_{i} [u_j]$ then we can see that for any
      $j$ the monomial
      $\partial_{i_j}[u_j]$ occurs only once in the sum obtained by taking
      $\partial_{i_j}$ from all Lie monomials in the linear  combination
      $\partial_{i} h$.
      So, the coefficient by
      $\partial_{i_j}[u_j]$ in
      $\partial_{i_j}h$ is equal to
      $\alpha_j$.

      We have shown that all derivatives of
      $g_1',g_2', \dots, g_n'$ are the polynomials with integer coefficients. Represent
      $g_1',g_2', \dots, g_n'$ as linear combinations of basis monomials. By the last paragraph, all coefficients
      by these monomials are integer.

      Since
      $\mathbf{C}$ and
      $\mathbf{D}$ are invertible so is their product. Again by
\cite{Um93} we obtain that the homomorphism
      $\mu:  M\to M$ defined by the rule
      $\mu(x_i)=g_i$ is an automorphism. So,
      $\{g_1,g_2,\dots, g_n\}$ is a set of free generators of
      $M$.

      Conversely, let
      $\{g_1,g_2,\dots,g_k\}$ be primitive system of elements. Then it can be included in a system of free generators
      $\{g_1,g_2,\dots, g_n\}$ of
      $M$. So,
      $\mu: M\to M$ defined by the rule
      $\mu(x_i)=g_i$ is an automorphism. Therefore, Lemma~%
\ref{automorph} implies
      $|\mathcal{J}(g_1,g_2,\dots, g_n)|=\pm 1$. Expanding this determinant
      by the last column we obtain
      $1$  as a linear combination of
      $(n-1)\times (n-1)$ minors constructed  on the first
      $n-1$ rows. Expanding these minors by the last columns we obtain a linear combination of
      $(n-2)\times (n-2)$ we obtain
      $1$  as a linear combination of
      $(n-2)\times (n-2)$ minors constructed  on the first
      $n-2$ rows and so on. Finally, we obtain an equality of the
      form
      $\sum_{i} f_im_i =\pm 1$, where
      $m_i$ are
      $k\times k$ minors constructed  on the first
      $k$ rows and
      $f_i\in \mathbb{Z}[X]$. Multiplying by
      $-1$ if necessary, we see that
      $1$ lies in the ideal of the ring
      $\mathbb{Z}[X]$ that is generated by
      $m_1,m_2,\dots$. Therefore this ideal coincides with
      $\mathbb{Z}[X]$.  This concludes the proof.
    \end{proof}
    \begin{llll}\label{unimod-prim}
      A  system of elements
      $\{g_1,g_2,\dots,g_k\}$
      ($1\leqslant k\leqslant n$) of the free  metabelian Lie ring
      $M$ is primitive iff for any positive integer numbers
      $p,q,m$ such that
      $m\geqslant 2$ the ideal generated by
      $k\times k$ minors of
      $\overline{\mathcal{J}}_{p,q,m}(g_1,g_2,\dots,g_k)$ coincides with
      $\mathbb{Z}_{p,q,m}[X]$.
    \end{llll}
    \begin{proof}
      If a system of elements
      $\{g_1,g_2,\dots,g_k\}$ is primitive then  by Theorem~%
\ref{equivprimunimod} the ideal generated by
      $k\times k$ minors of
      $\mathcal{J}(g_1,g_2,\dots,g_k)$ coincides with the entire ring
      $\mathbb{Z}[X]$. It is clear that for any ring
      $\mathbb{Z}_{p,q,m}[X]$ the natural homomorphism
      $\mathbb{Z}[X]\to \mathbb{Z}_{p,q,m}[X]$ is surjective.
      Therefore, the ideal generated by
      $k\times k$ minors of
      $\overline{\mathcal{J}}_{p,q,m}(g_1,g_2,\dots,g_k)$ coincides with
      $\mathbb{Z}_{p,q,m}[X]$.

      Conversely, let the system of elements
      $\{g_1,g_2,\dots, g_k\}$ be not primitive. Then, by Theorem~%
\ref{equivprimunimod} the ideal generated by
      $k\times k$ minors of
      $\mathcal{J}_{p,q,m}(g_1,g_2,\dots,g_k)$ does not coincide
      with
      $\mathbb{Z}[X]$. Therefore,
      $\mathbb{Z}[X]/I$ is non-trivial. Obviously, this ring is finitely generated. So,
\cite{Ba71} implies that there exists an ideal
      $J$ such that
      $R=\mathbb{Z}[X]/I+J$ is a non-trivial finite ring. Consider the images
      $\breve{x}_1, \breve{x}_2,\dots, \breve{x}_n$ of the elements
      $x_1,x_2,\dots, x_n$ under the natural homomorphism
      $\mathbb{Z}[X]\to R$. Since the ring
      $R$ is finite, for any
      $\breve{x}_i$ there are integers
      $p_i,q_i>0$ such that
      $\breve{x}_i^{p_i}=\breve{x}_i^{p_i+q_i}$. Take
      $p_i$ to be least possible and for each such
      $p_i$ take
      $q_i$ to be least possible. Moreover, there exists a
      positive integer number
      $m\geqslant 2$ such that
      $m=0$ in
      $R$. Take
      $m$ to be also least possible. Let
      $p=\max\{p_1,p_2,\dots, p_n\}$ and
      $q$ be the least common multiple of
      $q_1,q_2, \dots, q_n$.

      It is clear that the map
      $X \to R$ defined by the rule
      $x_i\to \breve{x}_i$ can be extended up to the homomorphism
      $\theta: \mathbb{Z}_{p,q,m}[X] \to R$. Indeed, since
      $\mathbb{Z}[X]$ is the free associative commutative
      $n$-generated ring and
      $\mathbb{Z}_{p,q,m}[X] = \mathbb{Z}[X]/I_{p,q,m}$ we are only left to show that
      $\theta$ is well-defined. It means that the relations generating
      $I_{p,q,m}$ go to zero by the action of
      $\theta$. Let us  verify this for all relations  generating
      $I_{p,q,m}$. We have
      $\theta(m)=m=0$ by the choice of
      $m$. Since
      $q$ is the least common multiple of
      $q_1,q_2,\dots q_n$ it divides all
      $q_i$.  Therefore,
      $q=q_i\widetilde{q}_i$ for some positive integer number
      $\widetilde{q}_i$
      ($i=1,2,\dots, n$). Consequently,
      \begin{eqnarray*}
         \theta(x_i^p(x_i^q-1)) & = & \breve{x}_i^p(\breve{x}_i^q-1)= \\
              & = & \breve{x}_i^{p-p_i}\breve{x}_i^{p_i}(\breve{x}_i^{q_i\widetilde{q}_i}-1)= \\
              & = & \breve{x}_i^{p-p_i}(\breve{x}_i^{q_i(\widetilde{q}_i-1)}+\breve{x}_i^{q_i(\widetilde{q}_i-2)}+\dots +\breve{x}_i^{q_i}+1))
                (\breve{x}_i^{p_i}(\breve{x}_i^{q_i}-1))=\\
              & = & 0
     \end{eqnarray*}

     Thus,
     $\theta$ is a homomorphism and it is clearly surjective. The natural homomorphism
     from
     $\mathbb{Z}[X]$ to
     $\mathbb{Z}_{p,q,m}[X]$ takes
     $k\times k$ minors of
     $\mathcal{J}(g_1,g_2,\dots,g_k)$ to the corresponding
     $k\times k$ minors of
     $\overline{\mathcal{J}}(g_1,g_2,\dots,g_k)$. Moreover, the ideal generated by
     these elements does not coincide with the entire ring
     $\mathbb{Z}_{p,q,m}[X]$. Indeed, on the one hand,
     $\theta$ is a surjective  homomorphism, but on the other hand, all
     $k\times k$ minors of
     $\overline{\mathcal{J}}(g_1,g_2,\dots,g_k)$ are in the kernel of
     $\theta$. Therefore, if the ideal generated by
     $k\times k$ minors of
     $\overline{\mathcal{J}}(g_1,g_2,\dots,g_k)$ coincided with
     $R$ then
     $R$ would be the trivial ring. We get a contradiction.
   \end{proof}

   \begin{llll}\label{abrprimmeas}
     If  the system of  elements
     $\{g_1,g_2,\dots, g_k\}$
     ($1\leqslant k\leqslant n$) of the free metabelian Lie  ring
     $M$ is uniformly distributed on the variety of metabelian Lie
     rings then the natural homomorphism from
     $M$ to
     $M/M'$ takes the system
     $\{g_1,g_2,\dots, g_k\}$ to a primitive system of elements on the variety of free abelian Lie
     rings.
   \end{llll}
   \begin{proof}
     Suppose that
     $g_i\in M'$. Let
     $A$ be an arbitrary finite abelian Lie ring. Let us choose a basis
     of
     $M$ such that all elements  in this basis are Lie monomials. Then represent
     $g_i$ as an integer valued linear combination of the elements of this basis. Obviously, any monomial in
     this  linear combination is a product of at least two generators of
     $M$. Therefore,
     $g_i(r_1,r_2,\dots,r_n)=0$ in
     $A$ for any
     $r_1,r_2,\dots r_n\in A$. But
     $g_i(x_1,x_2,\dots, x_n)$ is uniformly distributed on the variety of metabelian Lie
     rings. In particular, it is uniformly distributed on any
     finite abelian Lie ring. This is a contradiction. So,
     $g_i \in M\backslash M'$ for any
     $k=1,2,\dots,k$.

     Arguing as in the last paragraph we see that
     $g_i'(r_1,r_2,\dots,r_n)=0$ for any
     $r_1,r_2,\dots, r_n\in A$. Consequently, for any
     $r_1,r_2,\dots, r_n\in A$ we have
     $g_i(r_1,r_2,\dots,r_n)=\overline{g}_i(r_1,r_2, \dots r_n)$.
     Therefore, the system of elements
     $\{\overline{g}_1,\overline{g}_2,\dots,\overline{g}_k\}$ is uniformly distributed on
     $A$.

     Note that the image of
     $g_i$ by the natural homomorphism
     $M\to M/M'$ is equal to
     $\overline{g}_i$. Since we can consider an arbitrary the finite abelian Lie ring
     $A$ we obtain the system
     $\{\overline{g}_1,\overline{g}_2,\dots,\overline{g}_k\}$ is uniformly distributed
     on the variety of  abelian Lie ring. Since the multiplication on
     abelian Lie rings is trivial, we can consider an abelian Lie
     ring  as an additive abelian group. Finally,  Lemma~%
\ref{abgrlem} implies that since the system of elements
     $\{\overline{g}_1,\overline{g}_2,\dots,\overline{g}_k\}$  is uniformly distributed on the variety of
     abelian groups,  this system is primitive in the free abelian group
     $G$ generated by the set
     $X$. It means that this element is primitive
     in the free abelian Lie ring that is obtained from
     $G$ by adding on the multiplication in
     the trivial way. This is exactly what we needed.
   \end{proof}

   We need one more well known auxiliary lemma. We are going to give its proof for the completeness of reasoning.
   \begin{llll} \label{matrixsubst}
     Let
     $$
     S_1=
     \begin{pmatrix}
       s_1   & 0\\
       \tau_1& 0
     \end{pmatrix}, \
     S_2=
     \begin{pmatrix}
       s_2    & 0 \\
       \tau_2 & 0
     \end{pmatrix}, \dots,
     S_n=
     \begin{pmatrix}
       s_n      & 0\\
        \tau_n  & 0
     \end{pmatrix}
     $$
     be elements of the Lie ring
     $\mathcal{M}_{p,q,m}$. Then
     \begin{equation} \label{matprod}
       g(S_1,S_2,\dots, S_n)=
       \begin{pmatrix}
          \overline{g}(s_1,s_2\dots, s_n) & 0\\
          \tau'                 & 0
       \end{pmatrix},
      \end{equation}
      where
      $\tau'=\sum_{i=1}^n \tau_i\cdot \opartial_i g(s_1,s_2,\dots,s_n)$.
   \end{llll}
   \begin{proof}
      Let
      $g=[v]$. Without loss of generality we may assume that
      $[v]$ is a right-normed monomial, i.e. a monomial of the form
      $[\dots[x_{i_1},x_{i_2}],\dots, x_{i_k}]$. If
      $\ell([v])=1$ then the proof is trivial. Let
      $\ell([v])\geqslant 2$. Suppose that the statement is true for any monomials of lesser length.  We have
      $[v]=[[u],x_j]$. By induction hypothesis
      $$[u](S_1,S_2,\dots, S_n)=
      \begin{pmatrix}
         \overline{[u]}(s_1,s_2\dots, s_n) & 0\\
         \tau'_{[u]}              & 0
      \end{pmatrix},$$
      where
      \begin{equation} \label{tau}
        \tau'_{[u]}=\sum_{i=1}^n \tau_i\cdot \opartial_i [u](s_1,s_2,\dots,s_n).
      \end{equation}
      Let us notice that
      $$\overline{[u]}(s_1,s_2\dots, s_n)=
       \begin{cases}
         s_i, & \text{if } [u]=x_i,\ i=1,2,\dots n \\
         0,   & \text{if } \ell([u])\geqslant 2
       \end{cases}.$$
      For any
      $\opartial_i$ the following equality holds:
      \begin{equation}\label{difprod}
        \begin{split}
          \opartial_i[v]=& \opartial_i[[u],x_j]= \\
                        =& \eta\circ\varphi'(\partial'_i[u]\cdot x_j-\partial'_i x_j\cdot[u])=\\
                        =& \eta(\partial_i[u]\cdot x_j-\partial_i x_j\overline{[u]})=\\
                        =& \opartial_i[u]\cdot x_j-\delta_{ij}\overline{[u]}
        \end{split}
      \end{equation}
      where
      $\delta_{ij}$ is Kronecker delta.
      On the other hand,
      \begin{equation} \label{prodoftwomat}
        \left[
        \begin{pmatrix}
          \overline{[u]}(s_1,s_2,\dots,s_n) & 0 \\
          \tau'_{[u]}            & 0
        \end{pmatrix},
        \begin{pmatrix}
          s_j      & 0 \\
          \tau'_{j} & 0
        \end{pmatrix}
        \right]=
        \begin{pmatrix}
          0 & 0 \\
          \tau'_{[u]} s_j-\tau_j \overline{[u]}(s_1,s_2,\dots,s_n)& 0
        \end{pmatrix}.
      \end{equation}
      By
(\ref{matprod}),
 (\ref{tau}) and
(\ref{prodoftwomat}) we get
      $$\tau'=\biggl(\sum_{i=1}^n \tau_i\cdot \opartial_i
       [u](s_1,s_2,\dots,s_n))s_j-\tau_j\overline{[u]}(s_1,s_2,\dots, s_n\biggr).$$
      Taking into account
(\ref{difprod}), we obtain
      $\tau'=\sum_{i=1}^n \tau_i\cdot \opartial_i [u](s_1,s_2,\dots,s_n)$.
      So, for the case of a Lie monomial the proof is complete.
      If
      $g$ is a Lie polynomial, then the statement follows from the linearity of
      $\opartial_i$ and the linearity of the matrix addition.
   \end{proof}
   \begin{ttt}\label{main}
     A system of elements
     $\{g_1,g_2,\dots, g_k\}$
     $(1\leqslant k \leqslant n)$ of the free metabelian Lie algebra
     $M$ is primitive iff it is uniformly distributed on the variety of metabelian Lie algebras.
   \end{ttt}
   \begin{proof}
      Let a system of elements
      $\{g_1,g_2,\dots, g_k\}$ be primitive. Then it can be included in a set of generators of
      $M$. As we have shown, the property of a system of elements to be uniformly distributed on the variety of
      metabelian Lie algebras does not depend on a set of
      free generators chosen in
      $M$. Therefore, we may assume that
      $\{g_1,g_2,\dots,g_k\}\subseteq X$, for instance,
      $g_i=x_i$ for
      $i=1,2,\dots,k$. Let
      $R$ be an arbitrary finite metabelian Lie ring. Then for any elements
      $r_1,r_2,\dots, r_n$ we have
      $g_i(r_1,r_2,\dots, r_n)=r_i$. This means that for any
      $s_1,s_2,\dots, s_k$ the values of the polynomials
      $g_i$ are equal to the corresponding
      $s_i$ for
      $|R|^{n-k}$
      $n$-tuples
      $(r_1,r_2, \dots ,r_n)\in R^n$. Therefore, the system
      $\{g_1,g_2,\dots, g_k\}$ of elements in
      $M$ is uniformly distributed on
      $R$. Since finite metabelian Lie ring
      $R$ has been chosen arbitrarily we obtain that
      $\{g_1,g_2,\dots,g_k\}$ is uniformly distributed on the variety of metabelian Lie rings.

      Conversely let a system
      $\{g_1,g_2,\dots,g_k\}$ be uniformly distributed on the variety of metabelian Lie rings. Remind that
      the elements
      $\overline{g}_1,\overline{g}_2,\dots,\overline{g}_k$ are the images
      of the corresponding
      $g_1,g_2,\dots,g_k$ by the natural homomorphism
      $M\to M/M'$. By Lemma~%
\ref{abrprimmeas} the system
      $\{\overline{g}_1,\overline{g}_2,\dots,\overline{g}_k\}$ is primitive in
      $M/M'$. Easy to show, that in abelian Lie ring a linear basis is a set of generators.
      Therefore, without loss of generality it can be assumed that
      $\overline{g}_i\in X$ for
      $i=1,2,\dots,k$ (say,
      $\overline{g}_i=x_i$).

      Consider a ring
      $\mathcal{M}_{p,q,m}$. Lemma~%
\ref{matrixsubst} implies that for arbitrary elements
      $$
      S_1=
      \begin{pmatrix}
        s_1   & 0\\
        \tau_1& 0
      \end{pmatrix}, \
      S_2=
      \begin{pmatrix}
        s_2    & 0 \\
        \tau_2 & 0
      \end{pmatrix}, \dots,
      S_n=
      \begin{pmatrix}
        s_n      & 0\\
         \tau_n  & 0
      \end{pmatrix}
      $$
      of this ring the following equality holds:
      \begin{equation*} \label{matprod}
       g_i(S_1,S_2,\dots, S_n)=
       \begin{pmatrix}
          \overline{g}_i(s_1,s_2\dots, s_n) & 0\\
          \tau'_i & 0
       \end{pmatrix},
      \end{equation*}
      where
      $\tau_i'=\sum_{j=1}^n \tau_{j}\cdot \opartial_j g_i(s_1,s_2,\dots, s_n)$.
      Since the system
      $\{g_1,g_2,\dots,g_k\}$ is uniformly distributed on the variety of metabelian Lie algebras it is uniformly distributed on
      any ring, in particular, on
      $\mathcal{M}_{p,q,m}$. Consequently, for any system of matrices
      $(Q_1,Q_2,\dots,Q_k)$, where
      $Q_i=
       \begin{pmatrix}
         a_i            & 0 \\
         \widetilde{\tau}_i & 0
       \end{pmatrix}$, the system of equations
      $$
      \begin{cases}
         g_1(S_1,S_2,\dots, S_n)=Q_1,\\
         g_2(S_1,S_2,\dots, S_n)=Q_2,\\
         \dots\dots\dots\dots\dots\\
         g_k(S_1,S_2,\dots, S_n)=Q_k\\
      \end{cases}$$
      has exactly
      $|\mathcal{M}_{p,q,m}|^{n-k}$ solutions. Therefore the system of equations
      \begin{equation}\label{system}
        \begin{cases}
          \overline{g}_1(s_1,s_2\dots, s_n)=a_1, \\
          \overline{g}_2(s_1,s_2\dots, s_n)=a_2, \\
          \dots\dots\dots\dots\dots\\
          \overline{g}_k(s_1,s_2\dots, s_n)=a_k, \\
          \displaystyle{\sum_{j=1}^n \tau_{j}\cdot \opartial_j g_1(s_1,s_2,\dots, s_n)}=\widetilde{\tau}_1,\\
          \displaystyle{\sum_{j=1}^n \tau_{j}\cdot \opartial_j g_2(s_1,s_2,\dots, s_n)}=\widetilde{\tau}_2,\\
          \dots\dots\dots\dots\dots\dots\dots\dots\dots\\
          \displaystyle{\sum_{j=1}^n \tau_{j}\cdot \opartial_j g_k(s_1,s_2,\dots, s_n)}=\widetilde{\tau}_k.\\
        \end{cases}
      \end{equation}
      has
      $|\mathcal{M}_{p,q,m}|^{n-k}$ solutions.
      Here we say that the solution of system
(\ref{system}) is  a vector
      $$(s_1,s_2,\dots s_n,\tau_1,\tau_2,\dots, ,\tau_n)\in \mathbb{Z}_m^1[X]^n\times T^n$$
      satisfying all equations of this system.

      For any elements
      $s_1,s_2,\dots, s_n\in \mathbb{Z}^1_m[X]$ consider the map
      $\xi_{s_1,s_2,\dots,s_n}:T^n \to T^k$ defined by the rule
      \begin{equation*}
        \begin{split}
          \xi_{s_1,s_2,\dots,s_n}&(\tau_1,\tau_2,\dots, \tau_n)=\\
           =&\left(\sum_{j=1}^n \tau_j\cdot \opartial_i g_1(s_1,s_2,\dots,s_n),\dots,
              \sum_{j=1}^n \tau_j\cdot \opartial_i g_k(s_1,s_2,\dots, s_n)\right).
        \end{split}
      \end{equation*}
      Clearly, this map is a homomorphism of
      $\mathbb{Z}_{p,q,m}[X]$-modules. Therefore, cardinality of
      $\Ker\,\xi_{s_1,s_2,\dots,s_n}$ is not less than
      $|T|^{n-k}$. Moreover, if it is equal to
      $|T|^{n-k}$, then
      $\xi_{s_1,s_2,\dots,s_n}$ is a surjective homomorphism.
      Thus, for any
      $s_1,s_2,\dots,s_n$ and for any
      $(\widetilde{\tau}_1,\widetilde{\tau}_2,\dots,\widetilde{\tau}_k)$ in the image  of
      $\xi_{s_1,s_2,\dots,s_n}$ the system of the last
      $k$ equations of
(\ref{system}) has at least
      $|T|^{n-k}$ solutions
      $(\tau_1,\tau_2,\dots \tau_n)$.

      Consider the system of equations consisting of the first
      $k$ equations of
(\ref{system}) in the  case
      $a_i=x_i$ for
      $i=1,2,\dots, k$. Since
      $g_i=x_i+g_i'$, for
      $i=1,2,\dots,k$ the set of all solutions of this system
      consists of
      $n$-tuples of the form
      $(x_1,x_2,\dots,x_k,s_{k+1},\dots,s_n)$.
      So, cardinality of the set of
      $n$-tuples satisfying the system of the first
      $k$ equations of
(\ref{system}) is
      $|\mathbb{Z}^1_m[X]|^{n-k}$. As we have shown above, for each such
      $n$-tuple, there are at least
      $|T|^{n-k}$
      $n$-tuples
      $(\tau_1,\tau_2,\dots, \tau_n)$ satisfying the last
      $k$ equations of
(\ref{system}) for
      $\widetilde{\tau}_1=\widetilde{\tau}_2=\dots=\widetilde{\tau}_k=0$. We are left to notice that
      $|\mathcal{M}_{m,p,q}|=|\mathbb{Z}^1_m[X]|\cdot|T|$.
      Therefore, if for some
      $n$-tuple
      $(x_1,x_2,\dots,x_k,s_{k+1},\dots,s_n)$ there are more than
      $|T|^{n-k}$ solutions then the cardinality of the set of all solutions of system
(\ref{system}) is greater than
      $|\mathbb{Z}^1_m[X]|^{n-k}\cdot|T|^{n-k}=|\mathcal{M}_{m,p,q}|^{n-k}$.
      This contradicts to the assumption that the  system
      $\{g_1,g_2,\dots,g_k\}$ is uniformly distributed on
      $\mathcal{M}_{m,p,q}$.

      So, if
      $\widetilde{\tau}_1=\widetilde{\tau}_2=\dots=\widetilde{\tau}_k=0$ then for any
      $n$-tuple of the form
      $(x_1,x_2,\dots,x_k,s_{k+1},\dots,s_n)$ there are exactly
      $|T|^{n-k}$
      $n$-tuples
      $(\tau_1,\tau_2,\dots, \tau_n)$ satisfying the last
      $k$ equations of system
(\ref{system}). Consequently, for any
      $n$-tuple of the form
      $(x_1,x_2,\dots,x_k,s_{k+1},\dots,s_n)\in (\mathbb{Z}^1_m[X])^n$ the map
      $\xi_{x_1,x_2,\dots,x_k,s_{k+1},\dots, s_n}$ is a surjective homomorphism. So, for any such
      $n$-tuple the system of equations of the form
      $\displaystyle{\sum_{j=1}^n \tau_j\cdot \opartial_j g_i(x_1,x_2,\dots,x_{k},s_{k+1},\dots,s_n)}=\widetilde{\tau}_i$,
      $i=1,2,\dots,k$
      has a solution. In particular, the system of equations
      \begin{equation*}
        \begin{cases}
          \sum_{j=1}^n \tau_j\cdot \opartial_j g_1(x_1,x_2,\dots, x_n)=t_1,\\
          \sum_{j=1}^n \tau_j\cdot \opartial_j g_2(x_1,x_2,\dots, x_n)=t_2,\\
          \dots\dots\dots\dots\dots\dots \\
          \sum_{j=1}^n \tau_j\cdot \opartial_j g_k(x_1,x_2,\dots, x_n)=t_k,\\
        \end{cases}
      \end{equation*}
      has a solution.

      Let
      $\tau_j=\sum_{l=1}^n t_l\alpha_{jl}$ for
      $j=1,2,\dots, n$. Then, considering the coefficients by
      $t_1,t_2,\dots,t_k$ in the right-hand side and in the left-hand side  we obtain
      $$\sum_{j=1}^n \alpha_{j,l}\opartial_j g_a(x_1,x_2,\dots,x_n)=\delta_{l,a},$$
      $l=1,2,\dots,k$,
      $a=1,2,\dots, n$, and
      $\delta_{l,a}$  is Kronecker delta.

      Therefore,
      \begin{equation}\label{matjac}
        \begin{pmatrix}
          \alpha_{1,1} & \alpha_{2,1} & \dots & \alpha_{n,1}\\
          \alpha_{1,2} & \alpha_{2,2} & \dots & \alpha_{n,2}\\
          \dots        & \dots        & \dots &  \dots \\
          \alpha_{1,k} & \alpha_{2,k} & \dots & \alpha_{n,k}\\
        \end{pmatrix}
        \cdot
        \overline{\mathcal{J}}_{p,q,m}(g_1,g_2,\dots, g_k)=\mathbf{E}_{k}.
      \end{equation}
      Since
      $|\mathbf{E}_k|=1$ equality
(\ref{matjac}) implies that
      the unit  of
      $\mathbb{Z}_{p,q,m}[X]$ is a linear combination of
      $k\times k$ minors of
      $\overline{\mathcal{J}}(g_1,g_2,\dots,g_k)$. Therefore, the
      ideal generated by all
      $k\times k$ minors of
      $\overline{\mathcal{J}}(g_1,g_2,\dots,g_k)$ coincides with
      $\mathbb{Z}_{p,q,m}[X]$. Since this holds for any positive
      integers
      $p,q,m$, such that
      $m\geqslant 2$, Lemma~%
\ref{unimod-prim} implies that the system of elements
      $\{g_1,g_2,\dots, g_k\}$ is primitive.
  \end{proof}
  
\end{document}